\documentclass{elsarticle}
\usepackage{amsthm}
\usepackage{amssymb}
\usepackage{amsmath}
\usepackage{amscd}
\usepackage{color}

\newtheorem{thm}{Theorem}[section]
\newtheorem{cor}[thm]{Corollary}
\newtheorem{lem}[thm]{Lemma}
\newtheorem{prop}[thm]{Proposition}

\newtheorem{dfn}[thm]{Definition}

\newtheorem{rem}[thm]{Remark}

\makeatletter
 
 \@addtoreset{equation}{section}
\makeatother
\def\spmapright#1{\smash{%
 \mathop{\hbox to 1.3cm{\rightarrowfill}}
  \limits^{#1}}}

\begin{document}
\begin{frontmatter}
\title{\bf Jumping flatness and  Aluthge transform of recursive weighted shifts}
\tnotetext[t1]{The  authors are supported by the CeReMaR and the  Hassan II Academy of sciences. The last author is supported by African university of Sciences and technology-Abuja. Nigeria.}
\author[fsr]{Hamza El Azhar}
\ead{elazharhamza@gmail.com}
%-------------------------
\author[fsr]{Kaissar Idrissi}
\ead{kaissar.idrissi@gmail.com} 
%----------Author 1
\author[fsr]{El Hassan Zerouali\corref{cor1}}
\ead{elhassan.zerouali@um5.ac.ma}
\cortext[cor1]{Corresponding author}
\address[fsr]{Mohammed V University in Rabat, Rabat, Morocco.}

%{\footnote AMS subject class[2020]{Primary Primary 11B99; 44A60; 47B37 Secondary  30C15, 40A99}
%\footnote{}
\begin{abstract}

 We devote this paper  to Hamburger type weighted shifts.  We  give  in particular  an affirmative  answer to a problem concerning  subnormality of the Aluthge transform of Hamburger moment measures with finite support. we also extend the notion flatness,  “jumping flatness property”  introduced recently by Exner et all  for 
Hamburger-type weighted shift and provide  obtain several results related to the representing measure of such weighted shifts. 

\end{abstract}
\end{frontmatter}

%\maketitle

\section{Introduction}
Let us denote  $\mathcal{H}$ an infinite dimensional  Hilbert space and   let $\mathcal{L(H)}$ be  the space of all bounded linear operators on $\mathcal{H}$. An operator  $T \in \mathcal{L(H)}$ is normal if $TT^*=T^*T$,  is subnormal if it is the restriction of some normal operators and is hyponormal if $T^*T-TT^*\ge 0$. Here $T^*$ stands for the usual adjoint operator of $T$. The polar decomposition of an operator is given by the unique representation $T= U|T|$, where $|T| = (T^*T)^{\frac{1}{2}}$ and $U$ is a partial isometry satisfying $ker U = ker T$ and $ker U^*=ker T^*$. The Aluthge transform is then given by the expression 
$$\tilde T= |T|^{\frac{1}{2}}U|T |^{\frac{1}{2}}.$$
  The Aluthge  transform was introduced in  \cite{al} by Aluthge, in
 order to extend several  inequalities valid for hyponormal operators, and has received deep attention in the recent years.

 We consider below $\mathcal{H}=l^2(\mathbb{Z}_+)$ endowed by some  orthonormal basis $\{e_n\}_{n\in\mathbb{Z}_+}$. The forward  shift operator  $W_{\alpha}$  is defined on the basis  by $W_\alpha e_n=\alpha_ne_{n+1}$, where   $\alpha=\{\alpha_n\}_{n\geq 0}$ is a given  a  sequence of positive real numbers (called {\it weights}). We associate with $W_\alpha$ the moments sequence  obtained  by $$\gamma_0=1  \mbox{  and } \gamma_k\equiv\gamma_k(\alpha):=\alpha_0^2\alpha_1^2\cdots\alpha_{k-1}^2 \mbox{ for } k\geq 1.$$ 
 We will say that a sequence $\gamma$ admits a representing signed measure (called also a charge) supported in $K \subset {\mathbb R}$,  if 
 \begin{equation}\label{cmp}
\gamma_n = \int_Kt^nd\mu(t)\: \: \mbox{ for every } \: n\ge 0 \:\: \: \mbox{and} \: \: supp(\mu) \subset K.
\end{equation}
It is known since  1938 that every sequence is a  charge moment sequence supported in the real field.  \cite{polya1938indetermination}.

The weighted shift operator is bounded if and only if  $\|W_{\alpha}\|=sup_{n\geq0}\alpha_n<+\infty$.  It is clear  that $W_{\alpha}$ is never  normal  and that  $W_\alpha$ is hyponormal precisely when $\alpha_n$ is non decreasing. On the other hand,   Berger's Theorem says, $W_{\alpha}$  is a subnormal operator, if and only if, there exists a nonnegative Borelean measure $\mu$ (called {\it Berger measure}), which is a representing measure of $\{\gamma_n\}_{n\geq0}$ and such that $supp(\mu)\subset[0,\|W_{\alpha}\|^2]$. The latter is equivalent to the positivity of the   two Hankel matrices $(\gamma_{i+j})_{i,j \ge 0}$ and  $(\gamma_{i+j+1})_{i,j \ge 0}$. It is also known that  $\{\gamma_n\}_{n\geq0}$ admits a nonnegative  representing measure supported in ${\mathbb R}$ if and only if $(\gamma_{i+j})_{i,j \ge 0}$. Such a seqeuence willl said to be Hamburger sequence.

 We accord to the sequence $\gamma \equiv \{ \gamma_n \}_{n \in \mathbb{Z}_+}$ the following matrices
 \begin{equation}\label{test,6}
 M_{n}(k) = \left( \begin{matrix}
 \gamma_k            & \gamma_{k +1}            & \ldots & \gamma_{k +n}\\
 \gamma_{k +1}       & \gamma_{k +2}            & \ldots & \gamma_{k +n +1}\\
 \vdots              & \vdots                   & \ddots & \vdots\\
 \gamma_{k +n}       & \gamma_{k +n +1}         & \ldots & \gamma_{k +2n}
 \end{matrix} \right), \text{ for } n, k \in \mathbb{Z}_+.
 \end{equation}
 
  Recall the following definitions from \cite{Ex3}, 
\begin{dfn}
 A weighted shift $W_\alpha $ has property $H(n)$ [resp., property $\tilde H(n)]$ if $ M_{n}(k) \ge 0$ for all $k =
0, 2, 4, ... [resp., \tilde  M_n(k) \ge  0 $ for all $k = 1, 3, 5, ...]. $ And $W_\alpha $ has property $H(\infty) $ [resp., property $\tilde H(\infty)] $ if
it has property $H(n) $ [resp., property $\tilde H(n)$] for all $n \in {\mathbb Z}_+$. In particular, $W_\alpha $ is a Hamburger-type weighted
shift if $W_\alpha $ has property $H(\infty)$ and is subnormal if $W_\alpha $  has both $H(\infty)$ and $\tilde H(\infty)$.
\end{dfn}
  For a large familly of  shifts, a flatness phenomena occurs, when two successive weight are equal. More precisely 
  \begin{prop} Let $W_\alpha $ be a  weighted shift such that $\alpha_{n_0}=\alpha_{n_0+1}$ for some $n_0\ge 1$. We have 
  
  $(i)$ \:  (\cite[ Theorem\, 6]{stam} ). If $W_\alpha $ subnormal, then $\alpha$ is flat, i.e,   $\alpha_1=\cdots =\alpha_n =\cdots$

$(ii)$ \: (\cite[Corollary\, 6]{curt}). If $W_\alpha $ has property $H(2)$ and  $\tilde H(2)$  then $\alpha$ is flat.

$(iii)$ \: (\cite[Theorem\, 4.2]{Ex1}). If $W_\alpha $ has property $H(3)$   then $\alpha$ is flat.
\end{prop} 

Jumping flatness of weighted shifts was considered in some recent paper by Exner et all. It is defined as follows,
\begin{dfn}  A weighted shift  $W_\alpha $ with weight sequence $\alpha = (\alpha_n)_{n\ge 0}$ is said to have  
the jumping flatness property if $\alpha_n= \alpha_{n+2}$ for every $n\ge 1$.  In addition, $W_\alpha $  
has jumping flatness property  of type I, if $\alpha_0<\alpha_2$, and   has jumping
flatness property of type II, if $\alpha_0=\alpha_2$.
\end{dfn}
We outline the definition above to give the next extension of jumping flatness,

 \begin{dfn}
 A weighted shift $W_\alpha$, with weight sequence $(\alpha_n)_{n\ge 0}$, has the $k$-jumping flatness property if $\alpha_n = \alpha_{n +k}$ for every integer  $n \ge 1$.
 \end{dfn}

Clearly, flatness is $1$-jumping flatness property and the Jumping flatness property, introduced in \cite{Ex3}, coincides  with $2$-jumping flatness property. Moreover, subnormal weighted shifts with $k$-jumping property are flat, since their weights are non-decreasing. Hence, $k$-jumping flatness property is consistent only for non-subnormal weighted shifts.

The Aluthge transform $\tilde W_\alpha$ of a weighted shift  $W_\alpha$  is also a weighted shift, denoted below   $W_{\tilde \alpha}$ . Indeed, it is easy  to check that  $|W_\alpha| e_n=\alpha_ne_n$ and that $Ue_n=e_{n+1}$. It   follows then that
$$\tilde W_\alpha e_n= \sqrt{\alpha_n\alpha_{n+1}}e_{n+1}= {\tilde \alpha_n}e_{n+1}= W_{\tilde \alpha}e_n.$$
Notice also that ${\tilde \gamma}^2_n= \frac{1}{\alpha_0}\gamma_n \gamma_{n+1}$. In several recent  papers, the problem of subnormatiy of the aluthge transform of  weighted shifts  was considered. See \cite{Ex3}, for example.
 In \cite{Ex3}, the next question sas considered
 
 {\bf (HP) } Under what conditions is the Hamburger-type property of a weighted shift preserved under the Aluthge
transform?

A   special attention was devoted to weighted shifts with $2$-jumping flatness property. Since in this case , as it will be shown below, the representing measure has only two or 3 atoms, the next general problem is stated.

{\bf Problem 5.8.} Let  $W_\alpha$  be a weighted shift with the associated Hamburger moment measure $\mu  := \phi \delta_p + \psi\delta_r +\rho\delta_q $
for some $p > 0$ and $ -p <r<q$. Is it true that  ${\tilde W}_\alpha$  is subnormal if and only if
$r = 0$ and $ p = q?$\\
 
This paper is organised as follows, We solve first problem 5.8  by giving an affirmative answer in the next section. We also give several results concerning {\bf (HP)}.  Section 3 is devoted to $k-$ jumping flatness property for weighted shifts satisfying $H(n)$ for some adequate $n$. It is shown that a propagation phenomena occurs in this case. 
\section{ Weighted shifts with subnormal Aluthge transform}
Let  $W_\alpha$ be a weighted shift with  associated representing moment measure 
$$
\mu:=a\delta_{-p}+b\delta_{r}+c\delta_{q}
$$
for some $p >0$ and $-p <r<q$. The problem is to find conditions on the parameters $p, q$ and $r$ equivalent to ${\tilde W}_\alpha$ is  subnormal. We start with the next affirmative answer to Problem 5.8.
\begin{thm}
Let $W_\alpha$ be a weighted shift with the associated representing moment measure 
$$
\mu:=a\delta_{-p}+b\delta_{r}+c\delta_{q}
$$
for some $p >0$ and $-p <r<q$.
Then ${\tilde W}_\alpha$ is subnormal if and only if $r=0$ and $p=q$.
\end{thm}
\begin{proof}
   The following proposition that extends the subnormal case treated in \cite[Proposition 1.2]{Curto2019},  is the key point  of our proof:
    \begin{prop}\label{lemmeconvolution}
       Let $W_\alpha$ be a Hamburger weighted shift with  associated representing moment measure  $\mu$. Then  ${\tilde W}_\alpha$  admits a representing moment measure  if and only if  there exists a probability measure $\nu$ such that $\nu*\nu=\mu*t\mu$. Furthermore, ${\tilde W}_\alpha$ is subnormal if and only if $\nu$ is $\mathbb{R}^+$-supported.

Where $*$ denotes the multiplicative convolution 
$$
[\nu*\mu](A)=\int_{\mathbb{R}} \chi_A(xy) d\nu(x)d\mu(y)
$$
    \end{prop}
\begin{proof}[Proof of Proposition \ref{lemmeconvolution}]

        For the direct implication,  suppose  there exists a measure $\nu$ such that $\nu*\nu=\mu*t\mu$.  Then 
            \begin{align*}
                \alpha_0\tilde\gamma_n^2 &=\gamma_n \gamma_{n+1} =\left(\int_{\mathbb{R}}t^n d\mu(t)\right)\left(\int_{\mathbb{R}}t^{n+1} d\mu(t)\right)\\&=\left(\int_{\mathbb{R}}t^n d\mu(t)\right)\left(\int_{\mathbb{R}}s^{n} sd\mu(s)\right)=\int_{\mathbb{R}}\int_{\mathbb{R}} (st)^n d\mu(t)sd\mu(s)\\&=\int_{\mathbb{R}} u^n d(\mu*t\mu)(u)=\int_{\mathbb{R}} u^n d(\nu*\nu)(u)=\left(\int_{\mathbb{R}} u^n d\nu(u)\right)^2,
            \end{align*}
        and hence $\nu$ is   representing for  ${\tilde W_\alpha}$.\\
    Conversely, assume that  ${\tilde W_\alpha}$ admits a representing measure. Since $W_\alpha$ admits a representing measure,  then $W_{\tilde \alpha^2}=W_{\alpha S(\alpha)} $ has as representing measure $\mu*t\mu$. Where $S(\alpha)= (\alpha_{n+1})_n$.  Thus we have
$$
\int_{\mathbb{R}} u^n d(\nu*\nu)(u)=\left(\int_{\mathbb{R}} u^n d\nu(u)\right)^2=\int_{\mathbb{R}} u^n d(\mu*t\mu)(u)
$$
Since the weighted shift $W_\alpha$ is bounded,  we get the the support of $\mu$ and $\nu$ are compact \cite{Ex1}, so by Riesz representation theorem, we have $\nu*\nu=\mu*t\mu$.
   
        For the last point, since the measure $\nu$ is  Hamburger determinate we get that the Aluthge transform is subnormal if and only if $\nu$ is supported on $\mathbb{R}^+$.
    Sufficiency is obtained  by \cite[Theorem 5.7]{Ex3}.
        \end{proof}

    Assume now that $\tilde W_\alpha$ is subnormal. There exists a measure $\nu$ supported on $\mathbb{R}^+$, such that $\mu*t\mu=\nu*\nu$, in particular the support of $\mu*t\mu$ is in $\mathbb{R}^+$. Moreover,
\begin{equation}\label{conv2}
        \mu*t\mu=-pa^2\delta_{p^2}+rb^2\delta_{r^2}+qc^2 \delta_{q^2}+ab(r-p) \delta_{-pr}+ac(q-p)\delta_{-pq}+bc(q+r)\delta_{rq}
\end{equation}
    Since $q> 0$ (because weights are positive), $-pq< 0$, so $ac(q-p)=0$ this is equivalent to say that $p=q$. Hence \eqref{conv2} transforms to
\begin{equation}\label{conv2'}
        \mu*t\mu=p(c^2-a^2)\delta_{p^2}+rb^2\delta_{r^2}+ab(r-p) \delta_{-pr}+bc(p+r)\delta_{rp}.
\end{equation}
    If we assume that $r\neq 0$ then because of the term $ab(r-p) \delta_{-pr}$, we get $r=p=q$ which is impossible ($r<q$). Finally $p=q$ and $r=0$.
   \end{proof}
   We deduce the next corollary,
 \begin{cor}\label{cor3.2}
 Let $W_\alpha$ be a Hamburger recursively generated weighted shift and let $\mu $ be its  associated 
representing moment measure.  If ${\tilde W}_\alpha$ is of Hamburger type with associated measure $\nu$, then $\nu$ is discrete, Furthermore
$$
|supp(\nu)|=|supp(\mu)|.
$$
Where 
$$
|supp(\mu)|=\{ |\lambda| : \lambda \in supp(\mu)\}.
$$
 \end{cor}

    The previous proposition yields  a more general result of independent interest.
    
\begin{thm}\label{fib}
Let $W_\alpha$ be a non subnormal Hamburger recursively generated weighted shift and let $\mu=\sum\limits_0^ka_i\delta_{\lambda_i}$ be its  associated 
representing moment measure.  If for every $i\ne j, k$ satisfying $\lambda_i\le 0$, we have $\lambda_i\ne \lambda_j\lambda_k$, then ${\tilde W}_\alpha$  is not a Hamburger type shift.

\end{thm}
\begin{proof}  We have $
\mu:=\sum\limits_0^ka_i\delta_{\lambda_i}
$
for some $a_i> 0$,  and $\lambda_i\in {\mathbb R}$. We conclude by observing that
\begin{equation}\label{hamb}
\mu*t\mu = \sum\limits_{i,j =0}^ka_ia_j\lambda_j\delta_{\lambda_i\lambda_j} \end{equation}
Indeed, suppose  ${\tilde W}_\alpha$ is of Hamburger type,  We obtain  the measure of $\mu*t\mu$ is non negative, since $W_\alpha$ is non subnormal there is  $\lambda_i<0$. Then  the quantity  $a_i^2\lambda_i\delta_{\lambda_i^2} $ in the expression of $\mu*t\mu$,  contradicts $\mu$ is non negative.  
\end{proof}

Concerning the problem of preservation of Hamburger type property by the Aluthge transform  in  {\bf (HP)}, we have the following result.
\begin{thm}\label{fibo}
Let $W_\alpha$ be a Hamburger recursively generated weighted shift such that the associated measure has four atoms.
Then  ${\tilde W}_\alpha$  is not a Hamburger type shift.
\end{thm}
\begin{proof} 
We pose  $\mu = a_1\delta_{\lambda_1} +  a_2\delta_{\lambda_2}+  a_3\delta_{\lambda_3}+ a_4\delta_{\lambda_4}$, with $\lambda_1< \lambda_2< \lambda_3< \lambda_4$ and $a_i>0$ for $1\le i\le 4$. We will distinguish various cases. 
\begin{enumerate}
    \item $supp(\mu) =\{\lambda_1,\lambda_2,\lambda_3,\lambda_4\}\subset {\mathbb R}_+.$ We will get $W_\alpha$ is subnormal, and hence from \cite[Proposition 4.6]{brz} is not recursive. We conclude from Corollary \ref{cor3.2} that ${\tilde W}_\alpha$  is not a Hamburger type.
    \item $\lambda_1< 0 \le \lambda_2< \lambda_3< \lambda_4$. Assume that ${\tilde W}_\alpha$  is Hamburger type and let $\nu$ be the associated measure. 
    
    Writing $$\nu*\nu=\mu*t\mu =  a_1^2\lambda_1\delta_{\lambda_1^2} +  a_2^2\lambda_2\delta_{\lambda_2^2} +  a_3^2\lambda_3\delta_{\lambda_3^2} +  a_4^2\lambda_4\delta_{\lambda_4^2} +\sum\limits_{1\le i<j\le 4}  a_ia_j(\lambda_i+\lambda_j)\delta_{\lambda_i\lambda_j}, 
    $$
 we obtain   $\mu*t\mu$ is non negative, and then we derive, $-\lambda_1\le \lambda_2$. If $-\lambda_1< \lambda_2$, we will contradict again $\mu*t\mu$ is non negative, because of the term $a_1^2\lambda_1\delta_{\lambda_1^2}$ in  $\mu*t\mu$. Thus $\lambda_1 = -\lambda_2$. It will come 
$$\begin{array}{cl}
\mu*t\mu = & (a_2^2-a_1^2)\lambda_2\delta_{\lambda_2^2} +  a_3^2\lambda_3\delta_{\lambda_3^2} +  a_4^2\lambda_4\delta_{\lambda_4^2} +  a_1a_3(\lambda_3-\lambda_2)\delta_{-\lambda_2\lambda_3}  + a_1a_4(\lambda_4-\lambda_2)\delta_{-\lambda_2\lambda_4}\\ & +  a_2a_3(\lambda_3+\lambda_2)\delta_{\lambda_2\lambda_3}
+ a_2a_4(\lambda_4+\lambda_2)\delta_{\lambda_2\lambda_4}+ a_4a_3(\lambda_3+\lambda_4)\delta_{\lambda_3\lambda_4}
\end{array}
$$
In particular, 
\begin{equation*}
    \{\lambda_3^2,\lambda_4^2,\pm\lambda_2\lambda_3,\pm\lambda_2\lambda_4,\lambda_3\lambda_4\}\subset supp(\mu*t\mu)\subset \{\lambda_2^2, \lambda_3^2,\lambda_4^2,\pm\lambda_2\lambda_3,\pm\lambda_2\lambda_4,\lambda_3\lambda_4\}
\end{equation*}
 Using  Corollary  \ref{cor3.2}, we get
    $$\nu =  b_2\delta_{\lambda_2}+  b_3\delta_{\lambda_3}+ b_4\delta_{\lambda_4}
    +  c_2\delta_{-\lambda_2}+  c_3\delta_{-\lambda_3}+ c_4\delta_{-\lambda_4}$$
    with $b_i\ge 0$ and $c_i\ge 0$ for $2\le i\le 4$.
    
    From $\nu*\nu=\mu*t\mu$ we deduce that :
  \begin{equation}\label{supppp}
    \{\lambda_3^2,\lambda_4^2,\pm\lambda_2\lambda_3,\pm\lambda_2\lambda_4,\lambda_3\lambda_4\}\subset supp(\nu*\nu)\subset \{\lambda_2^2, \lambda_3^2,\lambda_4^2,\pm\lambda_2\lambda_3,\pm\lambda_2\lambda_4,\lambda_3\lambda_4\}
\end{equation}  
But,  
\begin{align*}
    \nu*\nu=& (b_2^2+c_2^2) \delta_{\lambda_2^2}+(b_3^2+c_3^2) \delta_{\lambda_3^2}+(b_4^2+c_4^2) \delta_{\lambda_4^2}+2[(b_2b_3+c_2c_3)\delta_{\lambda_2\lambda_3}+(b_3b_4+c_3c_4)\delta_{\lambda_3\lambda_4}\\&+(b_2b_4+c_2c_4)\delta_{\lambda_2\lambda_4}+(b_3c_2+b_2c_3)\delta_{-\lambda_2\lambda_3}+(b_3c_4+c_3b_4)\delta_{-\lambda_3\lambda_4}+(b_2c_4+c_2b_4)\delta_{-\lambda_2\lambda_4}\\&+b_2c_2\delta_{-\lambda_2^2}+b_3c_3\delta_{-\lambda_3^2}+b_4c_4\delta_{-\lambda_4^2}]
\end{align*}
We derive the next equations :
\begin{equation*}
    \left\{\begin{array}{l}
        b_4c_4=0   \\
         b_3c_4+c_3b_4=0\\
         b_2c_2=0
    \end{array} \right.
\end{equation*}
this is equivalent to

$$
         b_2c_2=0 \mbox{ and } (  (b_3,b_4)=(0,0)  \mbox{ or } 
         (c_3,c_4)=(0,0)).
$$

we derive   four different cases 
\begin{description}
    \item[I)-$(b_2,b_3,b_4)=(0,0,0)$] in this case $supp(\nu)\subset \mathbb{R}_-$, this implies that $supp(\nu*\nu)\subset \mathbb{R}_+$ which is impossible by \eqref{supppp}.
    \item[II)-$(c_2,c_3,c_4)=(0,0,0)$] Similarly in this case $supp(\nu)\subset \mathbb{R}_+$, this implies that $supp(\nu*\nu)\subset \mathbb{R}_+$ which is impossible by \eqref{supppp}.
    \item[III)-$(b_2,c_3,c_4)=(0,0,0)$] In this case $\nu=c_2\delta_{-\lambda_2}+b_3\delta_{\lambda_3}+b_4\delta_{\lambda_4}$, this implies that $\lambda_2\lambda_3\not\in supp(\nu*\nu)$ which contradict \eqref{supppp}.
    \item[IV)-$(c_2,b_3,b_4)=(0,0,0)$] Similarly in this case $\nu=b_2\delta_{\lambda_2}+c_3\delta_{-\lambda_3}+c_4\delta_{-\lambda_4}$, this implies that $\lambda_2\lambda_3\not\in supp(\nu*\nu)$ which contradict \eqref{supppp}.
\end{description}
We conclude from this discussion that $\tilde{W}_\alpha$ is not of Hamburger type.
\item $\lambda_1< \lambda_2 < 0 \le \lambda_3< \lambda_4$. Assume that ${\tilde W}_\alpha$  is Hamburger type and let $\nu$ be the associated measure. 

As in the second case, we drive that $\lambda_3\ge -\lambda_2$, and $\lambda_4\ge -\lambda_1$, if $\lambda_4>-\lambda_1$, because of the term $a_1^2\lambda_1\delta_{\lambda_1^2}$ in $\mu*t\mu$ this one will be non positive measure. Thus $\lambda_4=-\lambda_1$. Similarly, we get that $\lambda_3=-\lambda_2$. Write $\mu= a_1\delta_{-\lambda_4}+a_2\delta_{-\lambda_3} +
a_3\delta_{\lambda_3} + a_4\delta_{\lambda_4}$ and then 
$\nu= b_1\delta_{-\lambda_4}+b_2\delta_{-\lambda_3} +
b_3\delta_{\lambda_3} + b_4\delta_{\lambda_4}$. it will follow, 
\begin{align*}
    \mu*t\mu=&\lambda_4(a_4^2-a_1^2)\delta_{\lambda_4^2}+\lambda_3(a_3^2-a_2^2)\delta_{\lambda_3^2}\\&+(\lambda_3+\lambda_4)(a_3a_4-a_1a_2)\delta_{\lambda_3\lambda_4}+(\lambda_4-\lambda_3)(a_2a_4-a_1a_3)\delta_{-\lambda_3\lambda_4},
\end{align*}
and
\begin{align*}
    \nu*\nu=&(b_1^2+b_3^2)\delta_{\lambda_{4}^2}+(b_2^2+b_4^2)\delta_{\lambda_3^2}+2(b_1b_2+b_3b_4)\delta_{\lambda_3\lambda_4}+2(b_1b_3+b_2b_4)\delta_{-\lambda_3\lambda_4}\\&+2b_1b_4\delta_{-\lambda_4^2}+2b_2b_3\delta_{-\lambda_3^2}.
\end{align*}
Since $\{-\lambda_4^2,-\lambda_3^2\}\not \in supp(\mu)$, we get that
\begin{equation*}
    \left\{\begin{array}{l}
        b_1b_4=0   \\
        b_2b_3=0\\
    \end{array} \right.
\end{equation*}
As for the second case we have four cases, $(b_1,b_2)=(0,0)$, $(b_1,b_3)=(0,0)$, $(b_4,b_2)=(0,0)$, and $(b_4,b_3)=(0,0)$. The proof run similarly for the four cases.

For $(b_1,b_2)=(0,0)$,    from the equality $\nu*\nu=\mu*t\mu$ we derive that
 \begin{equation*}
    \left\{\begin{array}{l}
        \lambda_4(a_4^2-a_1^2)=b_3^2   \\
        \lambda_3(a_3^2-a_2^2)=b_4^2 \\
        (\lambda_3+\lambda_4)(a_3a_4-a_1a_2)=2b_3b_4\\
        (\lambda_4-\lambda_3)(a_2a_4-a_1a_3)=0
    \end{array} \right.
\end{equation*}   
Since $\lambda_3>0$ and $\lambda_4>0$. This implies that
 \begin{equation*}
    \left\{\begin{array}{l}
        (\lambda_3+\lambda_4)^2(a_3a_4-a_1a_2)^2=4pq(a_4^2-a_1^2)(a_3^2-a_2^2)\\
       a_2a_4-a_1a_3=0
    \end{array} \right.
\end{equation*} 
We get 
$$
(\lambda_3+\lambda_4)^2(a_3a_4-a_1a_2)^2=4pq(a_3a_4-a_1a_2)^2\iff (\lambda_3-\lambda_4)^2(a_3a_4-a_1a_2)^2=0
$$
Using the third equation we conclude that $b_3b_4=0$. Finally, $\mu*t\mu=0$, which is impossible by Proposition \ref{lemmeconvolution}.
\item $\lambda_1<  \lambda_2< \lambda_3< 0 \le \lambda_4$. Assume that ${\tilde W}_\alpha$  is Hamburger type and let $\nu$ be the associated measure. 

As in the second case, we drive that $\lambda_4\ge -\lambda_1$, if $\lambda_4>-\lambda_1$, because of the term $a_1^2\lambda_1\delta_{\lambda_1^2}$ in $\mu*t\mu$ this one will be non positive measure. Thus $\lambda_4=-\lambda_1$. Similarly, because of the term $a_2a_3(\lambda_2+\lambda_3)\delta_{\lambda_2\lambda_3}$ in $\mu*t\mu$ this one will be non positive measure.
\end{enumerate}
\end{proof}

\begin{rem} \begin{enumerate}
    \item Theorem \ref{fib} together with \cite[Corollary 4.11]{brz} shed some light on questions (HP) and (SP) in \cite{Ex3} concerning the preservation of subnormality and Hamburger type property by the Aluthge transform of weighted shifts.
\item Theorem \ref{fibo} allows to produce easily  a whole class of examples of Hamburger type shift for which the Aluthge transform is not of Hamburger type.
\end{enumerate}
\end{rem}
\section{$k$-Jumping flatness property for   weighted shifts}
In various research papers related to the subnormal completion problem  in one variable, the recursiveness plays a central role in the explicit calculation of the subnormal completion of weighted shifts (see \cite{brz, cf3}). A sequence is recursive  when it    satisfies following recursive relation,
 %%%%%%%%%%%%%%%%%%%%%%%
\begin{equation}\label{test,1}
 \gamma_{n+1}=a_0\gamma_{n}+a_1\gamma_{n-1}+\cdots+a_{r-1}\gamma_{n-r+1} \: \: \mbox{ for every } \:n\geq r,
\end{equation}
%%%%%%%%%%%%%%%%%%%%%%%%%%
where the coefficients $a_0,a_1,\cdots , a_{r-1}$  are some fixed numbers and $\gamma_0, \gamma_1,\cdots,\gamma_{r-1}$ are the initial conditions. When $\gamma$ is recursive, we say that the associated weighted shift $W_\alpha$ is {\it recursively generated weighted shift}.  The polynomial $P(z) =z^{r}-a_{0}z^{r-1}-...-a_{r-2}z-a_{r-1}$, is said to be generating for $\gamma$.
 {\it  The characteristic polynomial} of $\gamma$ is the unique  generating polynomial  $P_\gamma(z)$  with minimal degree. It is usefully applied in the determination
of the explicit expression of the general term $\gamma_n$, by considering the
(characteristic) roots
$ \lambda_{1},\ \lambda_{2},\ ...,\ \lambda_{s}$   of $P_\gamma(z)$, with  multiplicities $m_{1},\,
m_{2}\, ...,\, m_{s}$ (respectively).\\
For weighted shifts, we have the following
\begin{prop} Let $W_\alpha $  be recursively generated weighted shift with characteristic polynomial $P_\gamma$  . Then, the following are equivalent
\begin{enumerate}
    \item  $P_\gamma$ has only simple roots $\lambda_1, \cdots, \lambda_r$.
\item $W_\alpha $  admits a representing measure $\mu = c_1\delta_{\lambda_1} + \cdots + c_r\delta_{\lambda_r}$
for some suitable constants $c_1, \cdots, c_r.$
\end{enumerate}
\end{prop}

Notice in passing that
\begin{prop} Let $W_\alpha $  be a
 weighted shift with jumping flatness property and set $p=\alpha_1\alpha_2$, Then $\gamma$ satisfies the next recursive relation $\gamma_{n+3} = p^2\gamma_{n+1} $ for every $n\in{\mathbb Z}_+$. In particular, we have
\begin{enumerate}
    \item If $W_\alpha $ has jumping flatness of type I, then $ P_\gamma(X)= X(X-p)(X+p)$
    \item If $W_\alpha $ has jumping flatness of type II, then $ P_\gamma(X)= (X-p)(X+p)$
\end{enumerate}

\end{prop}

We deduce the next corollary
\begin{cor}Let $W_\alpha $  be a
 weighted shift with jumping flatness property. Then $W_\alpha $ admits a fnite atomic  representing measure $\mu$. Moreover, 
 \begin{enumerate}
    \item If $W_\alpha $ has jumping flatness of type I, then $\mu = c_{-1}\delta_{-p} +c_0\delta_0+c_1\delta_p$
    \item If $W_\alpha $ has jumping flatness of type II, then  then $\mu = c_{-1}\delta_{-p} +c_1\delta_p.$
\end{enumerate}
for some suitable constants $c_{-1}, c_0$ and $c_1$ and with $p=\alpha_1\alpha_2$.
\end{cor}

  Without loss of generality, we assume in the sequel that  $\gamma \equiv \{ \gamma_n \}_{n \in \mathbb{Z}_+}$ is a recursive sequence of order $r$
 satisfying \eqref{test,1} with   $a_{r-1} \neq 0$,
and  $\gamma_0, \gamma_1, \ldots, \gamma_{r-1}$ are the initial conditions.  
 Note that the Equality \eqref{test,1} yields that
  \begin{equation}\label{test,2}\begin{aligned}
    \gamma_{n -r +1} &= \frac{1}{a_{r-1}} \gamma_{n +1} -\frac{a_0}{a_{r-1}} \gamma_{n} -\ldots -\frac{a_{r-2}}{a_{r-1}} \gamma_{n-r+2} \\
                     &= \sum\limits_{t=1}^r a'_t \gamma_{n -r +1 +t}.
   \end{aligned}
  \end{equation}

 \begin{lem}\label{test,lem-1}
  Let $\gamma \equiv \{ \gamma_n \}_{n \in \mathbb{Z}_+}$ be as in \eqref{test,1}. If $M_{r-1}(2n_0) \geq 0$ for some integer $n_0 \geq 0$,
 then $M_\infty(\gamma) \geq 0$.
 \end{lem}

 \begin{proof}
   It suffice to show that if we have  $M_{r-1}(2n_0) \geq 0 $ then $ M_{r}(2n_0 -2) \geq 0$
  and $M_{r}(2n_0 +2) \geq 0$. To this aim, consider $\mathbf{x} = (x_0, x_1, \ldots, x_r) \in \mathbb{R}^{r+1}$. We have
  \begin{equation*}\begin{aligned}
    \mathbf{x}^T M_r(2n_0 -2) \mathbf{x} &= \sum\limits_{i, j=0}^r x_j \gamma_{2n_0 -2 +i +j} x_i\\
    &= \sum\limits_{i, j=1}^r x_j \gamma_{2n_0 -2 +i +j} x_i
       +\sum\limits_{i=0}^r x_0 \gamma_{2n_0 -2 +i} x_i
       +\sum\limits_{j=0}^r x_j \gamma_{2n_0 -2 +j} x_0
       +x_0 \gamma_{2n_0 -2 +i +j} x_0.\\&= \sum\limits_{i, j=1}^r x_j \gamma_{2n_0 -2 +i +j} x_i
       +2\sum\limits_{i=0}^r x_0 \gamma_{2n_0 -2 +i} x_i
       +x_0 \gamma_{2n_0 -2 +i +j} x_0   \end{aligned}
  \end{equation*}
  By virtue of \eqref{test,2}, one have
  \begin{equation*}\begin{aligned}
    \sum\limits_{i=1}^r x_0 \gamma_{2n_0 -2 +i} x_i &= \sum\limits_{i=1}^r x_0 (\sum\limits_{t=1}^r a'_t \gamma_{2n_0 -2 +i +t}) x_i
                                                      = \sum\limits_{i, j=1}^r x_0 a'_j \gamma_{2n_0 -2 +i +j} x_i;\\
    x_0 \gamma_{2n_0 -2} x_0 &= x_0( \sum\limits_{t=1}^r a'_t \gamma_{2n_0 -2 +t}) x_0 = x_0( \sum\limits_{i=1}^r a'_i \sum\limits_{j=1}^r a'_j \gamma_{2n_0 -2 +i +j}) x_0\\
                              &= \sum\limits_{i, j=1}^r x_0 a'_j \gamma_{2n_0 -2 +i +j} x_0 a'_i.
   \end{aligned}
  \end{equation*}
  Let $\mathbf{y} = (y_1, y_2, \ldots, y_{r-1}) \in \mathbb{R}^{r+1}$ be given by $y_i = 2x_0 a'_i +x_i$ for $j=0, 1, \ldots, r-1$.
  It follows from above  that $\mathbf{x}^T M_r(2n_0 -2) \mathbf{x} = \mathbf{y}^T M_{r-1}(2n_0) \mathbf{x}$. Since $M_r(2n_0 -2)$ is positive semidefinite, then so is  $M_{r-1}(2n_0)$.

  To show that $M_{r-1}(2n_0) \geq 0 \Rightarrow M_{r}(2n_0 +2) \geq 0$ one repeats the above proof and using  Equality \eqref{test,1} instead of  Equality \eqref{test,2}. This completes the proof.

 \end{proof}

 \begin{thm}\label{test,th-1}
   Let $W_\alpha$ be a weighted shift with property $H([\frac{3k}{2}] +1)$. If there exists  $n_0 \in \mathbb{Z}_+$ such that 
   $$\alpha_{n_0 +j} = \alpha_{n_0 +k +j} \: \: \: \: \: \: \: \: \: \: \: \: \: \; \text{ for }j=0, 1, \ldots, k -1,$$
then $W_\alpha$ has  $k$-jumping flatness property.
 \end{thm}

 \begin{proof}
   (Outer propagation) Set $m_0 = 2[\frac{n_0}{2}]$.
   We have $M_{[\frac{3k}{2}] +1}(m_0) \geq 0$, then Smul'jan's theorem yields
  \begin{equation}\label{test,8}
    \gamma_{m_0 +[\frac{3k}{2}] +1 +i} = \sum_{j=0}^{[\frac{3k}{2}]} a_j \gamma_{m_0 +[\frac{3k}{2}] +i -j} \qquad ( i=0, 1 ,\dots, [\frac{3k}{2}] )
  \end{equation}
  for some real numbers $a_0, a_1, \dots, a_{[\frac{3k}{2}]}$. Let us consider the recursive sequence $\tilde{\gamma} = \{ \tilde{\gamma}_i \}_{i \in \mathbb{Z}_+}$ defined as follows
  \begin{equation}\label{test,7}
    \begin{cases}
      \tilde{\gamma}_l &= \gamma_l, \qquad l= m_0, m_0 +1, \ldots, m_0 +[\frac{3k}{2}]; \\
      \tilde{\gamma}_{[\frac{3k}{2}] +i +1} &= a_0 \tilde{\gamma}_{[\frac{3k}{2}] +i} +a_1 \tilde{\gamma}_{[\frac{3k}{2}] +i -1} +\ldots
      +a_{[\frac{3k}{2}]} \tilde{\gamma}_{i} \text{ for all } i \in \mathbb{Z}_+.
    \end{cases}
  \end{equation}
 Clearly, Formulas  \eqref{test,8} and \eqref{test,7} yield
  \begin{equation}\label{test,3}
 \tilde{\gamma}_i = \gamma_i \text{ for all } i= m_0, m_0+1, \ldots, m_0 +2[\frac{3k}{2}] +1.
  \end{equation}
  Hence $M_{[\frac{3k}{2}]}(\tilde{\gamma})(m_0) = M_{[\frac{3k}{2}]}(\gamma)(m_0)$. Since $M_{[\frac{3k}{2}]}(\gamma)(m_0) \geq 0$, then Lemma \ref{test,lem-1} implies that $M_\infty(\tilde{\gamma}) \geq 0$. Also, the sequence $\tilde{\gamma}$ is recursive, and then, according to \cite[Theorems 3.1 and 3.9]{cf91}, $\tilde{\gamma}$ has a finitely atomic  representing measure, say $\mu = \sum_{t=0}^{r-1} \rho_t \delta_{\lambda_t}$. In symbols
  \begin{equation}\label{test,4}
    \tilde{\gamma}_i = \int_\mathbb{R} x^i d\mu = \sum_{t=0}^{r-1} \rho_t \lambda_t^i.
  \end{equation}
 We have $
    \frac{\gamma_{n_0 +1 +j}}{\gamma_{n_0 +j}} = \displaystyle{ \alpha_{n_0 +j} =  \alpha_{n_0 +k +j} } =  \frac{\gamma_{n_0 +k +1 +j}}{\gamma_{n_0 +k +j}} \text{ for all } j = 0, 1, \ldots, k -1$, 
  then  \begin{equation}\label{test,31} \gamma_{n_0}\gamma_{n_0 +2k} = \gamma_{n_0 +k}^2.  \end{equation} By using \eqref{test,3}, \eqref{test,4} and \eqref{test,31}, one obtain
  \begin{equation*}
    \left( \sum\limits_{t=0}^{r-1} \rho_t (\lambda_t^k)^{n_0} \right) \left( \sum\limits_{t=0}^{r-1} \rho_t (\lambda_t^k)^{n_0 +2} \right)
    =\left(\sum\limits_{t=0}^{r-1} \rho_t (\lambda_t^k)^{n_0 +1} \right)^2.
  \end{equation*}
That implies
\begin{equation}\label{test,9}
\sum\limits_{0\leq i< j\leq r-1} \rho_i \rho_j (\lambda_i^k \lambda_j^k)^{n_0 +1} (\lambda_i^k -\lambda_j^k)^2 = 0,
\end{equation}
which gives $\lambda_i^k = \lambda_j^k$ whenever $\lambda_i \lambda_j \neq 0$.

Hence
\begin{equation}\label{test,10}
\mu = \rho_0 \delta_0 +\rho_1 \delta_{\lambda} +\rho_2 \delta_{-\lambda} \qquad \text{ with } \lambda >0.
\end{equation}
Note in passing that if $k$ is odd then $\rho_2 = 0$. Thereby,
\begin{equation*}
\tilde{\gamma}_0 = \rho_0 +\rho_1 +\rho_2 \text{ and } \tilde{\gamma}_n = \rho_1 (\lambda)^n +\rho_2(-\lambda)^n \text{ for all integer } n \geq 1.
\end{equation*}
For all $ j = 0, 1, \ldots, k -1$, the Equality \eqref{test,3} yields
 \begin{equation*}\begin{aligned}
   \alpha_{n_0 +j}^2 = \alpha_{n_0 +k +j}^2 &= \frac{\gamma_{n_0 +k +j +1}}{\gamma_{n_0 +k +j}} = \frac{\tilde{\gamma}_{n_0 +k +j +1}}{\tilde{\gamma}_{n_0 +k +j}}
   = \frac{\rho_1 \lambda^{n_0 +k +j +1} -\rho_2 \lambda^{n_0 +k +j +1}}{\rho_1 \lambda^{n_0 +k +j} -\rho_2 \lambda^{n_0 +k +j}}\\
   &= \frac{\rho_1 \lambda^{n_0 +2k +j +1} -\rho_2 \lambda^{n_0 +2k +j +1}}{\rho_1 \lambda^{n_0 +2k +j} -\rho_2 \lambda^{n_0 +2k +j}}
   = \frac{\tilde{\gamma}_{n_0 +2k +j +1}}{\tilde{\gamma}_{n_0 +2k +j}}
   = \frac{\gamma_{n_0 +2k +j +1}}{\gamma_{n_0 +2k +j}}\\
   &= \alpha_{n_0 +2k +j}^2.
   \end{aligned}
 \end{equation*}
Continuing this process, we obtain
\begin{equation}\label{test,11}
  \alpha_{n_0 +j} = \alpha_{n_0 +k +j} = \alpha_{n_0 +2k +j} = \alpha_{n_0 +3k +j} = \ldots,\qquad \text{ for } j= 0, 1 , \ldots, k-1.
\end{equation}

 (\textit{Inner propagation}) Now, let $n'_0$ be the smallest integer such that $n_0 \leq k n'_0$. According to \eqref{test,11}, we have
 $$
 \alpha_{k n'_0 +j} = \alpha_{(k+1) n'_0 +j}, \qquad j=0, 1, \ldots, k-1.
 $$
 Setting $m'_0 = \left[\frac{2(k-1)n'_0}{2}\right]$. The property $H([\frac{3k}{2}] +1)$ implies $M_{[\frac{3k}{2}] +1}(m'_0) \geq 0$. By using Smul'Jan's theorem, one obtain
 \begin{equation}\label{test,12}
    \gamma_{m'_0 +[\frac{3k}{2}] +1 +i} = \sum_{j=0}^{[\frac{3k}{2}]} b_j \gamma_{m'_0 +[\frac{3k}{2}] +i -j} \qquad ( i=0, 1 ,\dots, [\frac{3k}{2}] ),
  \end{equation}
where $b_0, \ldots, b_{[\frac{3k}{2}]}$ are real numbers.\\
Consider the recursive sequence $\hat{\gamma} = \{ \hat\gamma_i\}_{i \in \mathbb{Z}_+}$, defined by
\begin{equation}\label{test,13}
    \begin{cases}
      \hat{\gamma}_l &= \gamma_l, \qquad l= m'_0, m'_0 +1, \ldots, m'_0 +[\frac{3k}{2}]; \\
      \tilde{\gamma}_{[\frac{3k}{2}] +i +1} &= b_0 \tilde{\gamma}_{[\frac{3k}{2}] +i} +b_1 \tilde{\gamma}_{[\frac{3k}{2}] +i -1} +\ldots
      +b_{[\frac{3k}{2}]} \tilde{\gamma}_{i} \text{ for all } i \in \mathbb{Z}_+.
    \end{cases}
  \end{equation}
Remark that $\hat\gamma_i =\gamma_i$ for all $i= m'_0, \ldots, m'_0 +2[\frac{3k}{2}] +1$.\\
In a similar manner as in the proof of the "\textit{outer propagation}", we get
 $$
 \alpha_{(k-1)n'_0 +j} =  \alpha_{k n'_0 +j} (=  \alpha_{(k+1)n'_0 +j}) \qquad \text{ for } j = 0, \ldots, k-1.
 $$ 
Repeating the same process, we obtain the desired result. This completes the proof.
 \end{proof}

Note that the above proof furnishes  more interesting results. Indeed, if the hypothesis of Theorem \ref{test,th-1} are verified and $k$ is an odd integer. Then, as observed above, the measure given in \eqref{test,10} is represented by  $\mu = \rho_0 \delta_0 +\rho_1 \delta_{\lambda} $ where $\lambda$ is a positive number. Therefore, via \eqref{test,3}, one has 
$\alpha_{n_0 +1} = \alpha_{n_0 +2}$. By applying Theorem \ref{test,th-1}, the operator $W_\alpha$ is subnormal and has  1-jumping flatness property. 

Now, let us assume that the  hypothesis of Theorem \ref{test,th-1} are verified and $k$ is an even integer. In a similar way, one obtain $\alpha_{n_0} = \alpha_{n_0 +2}$ and $\alpha_{n_0 +1} = \alpha_{n_0 +3}$.
Using again Theorem \ref{test,th-1}, the operator $W_\alpha$ is of Hamburger-type and has  2-jumping flatness property.

\begin{thm}
    Let  $W_\alpha$ be a  weighed shift with  property $H([\frac{3k}{2}] +1)$ such that there  exists  $n_0 \in \mathbb{Z}_+$satisfying
   $$\alpha_{n_0 +j} = \alpha_{n_0 +k +j} \: \: \: \: \: \: \: \: \: \: \: \: \: \; \text{ for }j=0, 1, \ldots, k -1.$$ Then
   \begin{itemize}
     \item if $k$ is odd, then $W_\alpha$ is subnormal and is flat (has the $1$-jumping flatness property);
     \item if $k$ is even, then $W_\alpha$ is a Hamburger-type operator and has a 2-jumping flatness property.
   \end{itemize}
\end{thm}
\begin{rem}
In the previous theorem, for weighted shifts with  $k=2$, we have $H(n)=H([\frac{3k}{2}] +1)=H(4)$ which coincides with  the condition $H(4)$ as assumed in \cite{Ex3}.  It is natural to ask if the property $H([\frac{3k}{2}] +1)$ above  can be weakened to $H((\phi(n))$ for some $\phi(n) < [\frac{3k}{2}] +1$.
\end{rem}

\end{document}